\documentclass[12pt,reqno]{amsart}
\usepackage{amssymb}
\usepackage{amscd}
\usepackage{hyperref}

\usepackage{amsxtra}

\usepackage{hyperref}

\usepackage{esint}

\usepackage[shortlabels]{enumitem}
\usepackage{cleveref}

\crefrangeformat{enumi}{items #3#1#4--#5#2#6}

\theoremstyle{plain}

\newtheorem{thm}{Theorem}[section]
\newtheorem{lem}[thm]{Lemma}
\newtheorem{cor}[thm]{Corollary}
\newtheorem{prop}[thm]{Proposition}

\theoremstyle{definition}

\newtheorem{defn}[thm]{Definition}

\theoremstyle{remark}

\newtheorem{rem}[thm]{Remark}

\crefname{thm}{theorem}{theorems}
\crefname{lem}{lemma}{lemmas}
\crefname{cor}{corollary}{corollaries}
\crefname{prop}{proposition}{propositions}
\crefname{mainthm}{theorem}{theorems}
\crefname{maincor}{corollary}{corollaries}

\crefname{defn}{definition}{definitions}
\crefname{conj}{conjecture}{conjectures}
\crefname{example}{example}{examples}
\crefname{exercise}{exercise}{exercises}
\crefname{prob}{problem}{problems}
\crefname{quest}{question}{questions}

\crefname{rem}{remark}{remarks}
\crefname{claim}{claim}{claims}
\crefname{axiom}{axiom}{axioms}
\crefname{hyp}{hypothesis}{hypotheses}
\crefname{notation}{notation}{notations}
\crefname{case}{case}{cases}

\numberwithin{equation}{section}

\usepackage{color}
\definecolor{darkgreen}{cmyk}{1,0,1,.2}
\definecolor{m}{rgb}{1,0.1,1}

\newdimen\theight
\def\TeXref#1{%
             \leavevmode\vadjust{\setbox0=\hbox{{\tt
                     \quad\quad  {\small \textrm #1}}}%
             \theight=\ht0
             \advance\theight by \lineskip
             \kern -\theight \vbox to
             \theight{\rightline{\rlap{\box0}}%
             \vss}%
             }}%

\begin{document}
\vskip .5cm
\title[ Biharmonic Maps on  Foliations]
{Generalized Chen's conjecture for biharmonic maps on foliations}

\author[X. S. Fu]{Xueshan Fu}
\address{Department of Mathematics\\
         Shenyang  University of  Technology\\
         Shenyang 110870\\
         People Republic of China}
\email{xsfu@sut.edu.cn}

\author[S. D. Jung]{Seoung Dal Jung}
\address{Department of Mathematics\\
         Jeju National University\\
         Jeju 63243\\
         Republic of Korea}
\email{sdjung@jejunu.ac.kr}

\thanks{The second author was supported by the 2022 scientific promotion program funded by Jeju National University.}

\subjclass[2010]{53C12; 53C43; 58E20.}
\keywords{Riemannian foliation, Transversally harmonic map,  Transversally biharmonic map, $(\mathcal F,\mathcal F')$-harmonic map, $(\mathcal F,\mathcal F')$-biharmonic map,  Generalized Chen's conjecture}

\begin{abstract}
In this paper,  we prove the generalized Chen's conjecture  for $(\mathcal F,\mathcal F')$-biharmonic map, which is a critical point of the transversal bienergy functional.
\end{abstract}
\maketitle
\section{Introduction}
On a Riemannian geometry,  harmonic maps  play a central role to study the geometric properties. They are critical points of the energy functional $E(\phi)$ for smooth maps $\phi: (M,g)\to (M',g')$, where
\begin{align*}
E(\phi) =\frac12\int_M |d\phi|^2 \mu_M,
\end{align*}
where $\mu_M$ is the volume element.  It is well known that  harmonic map is a solution of the Euler-Largrange equation $\tau(\phi)=0$, where $\tau(\phi)={\rm tr}_g (\nabla d\phi)$ is the tension field. 

In 1983, J. Eells and L. Lemaire  extended the notion of harmonic map to biharmonic map, which is a critical points of the bienergy functional $E_2(\phi)$, where
\begin{align*}
E_2(\phi) = \frac12\int_M |\tau(\phi)|^2 \mu_M.
\end{align*}
It is well-known \cite{JI} that harmonic maps are always biharmonic. But the converse is not true. At first, B.Y. Chen \cite{CH} raised so called Chen's conjecture and later, R. Caddeo et.al. \cite{CMP} raised the generalized Chen's conjecture. That is,

{\bf Generalized Chen's conjecture:} {\it Every biharmonic submanifold of a Riemannian manifold of non-positive curvature must be harmonic.}

About the generalized Chen's conjecture, Nakauchi et. al. \cite{NUG} showed the following.

\begin{thm} \cite{NUG}  Let $(M,g)$ be a complete Riemannian manifold and $(M',g')$ be of non-positive sectional curvature.  Then

(1) every biharmonic map $\phi:M\to M'$ with finite energy and finite bienergy must be harmonic.

(2) In the case $Vol(M)=\infty$, every biharmonic map with finite bienergy  is harmonic.

\end{thm}

Now, we study the generalized Chen's conjecture  for biharmonic maps on foliated Riemannian manifolds and extend Theorem 1.1 to foliations.  Let $(M,g,\mathcal F)$ and $(M',g',\mathcal F')$ be the foliated Riemannian manifolds. Let $\phi:M\to M'$ be a smooth foliated map, that is, map preserving the leaves.  Then $\phi$ is said to be {\it $(\mathcal F,\mathcal F')$-harmonic map} \cite{DT} if  $\phi$ is a critical point of the transversal energy $E_B(\phi)$, which is given by
\begin{align*}
E_B(\phi) =\frac12\int_M |d_T\phi|^2\mu_M,
\end{align*}
where $d_T\phi =d\phi|_Q$ is the differential map of $\phi$ restricted to the normal bundle $Q$ of $\mathcal F$. From the first variational formula for the transversal energy functional \cite{JJ1}, it is trivial that $(\mathcal F,\mathcal F')$-harmonic map is a solution of $\tilde\tau_b(\phi):= \tau_b(\phi) +d_T\phi(\kappa_B)=0$, where $\tau_b(\phi)=tr_Q (\nabla_{tr}d_T\phi)$ is the transversal tension field and $\kappa_B$ is the basic part of the mean curvature form $\kappa$ of $\mathcal F$. 
 
The map $\phi$ is said to be {\it $(\mathcal F,\mathcal F')$-biharmonic map} if $\phi$ is a critical point of the transversal bienergy functional $\tilde E_{B,2} (\phi)$, where
\begin{align*}
\tilde E_{B,2}(\phi) =\frac12\int_M |\tilde \tau_b(\phi)|^2 \mu_M.
\end{align*}
By the first variation formula for the transversal bienergy functional $\tilde E_{B,2}(\phi)$ (Theorem 3.7), we know that $(\mathcal F,\mathcal F')$-harmonic map is always $(\mathcal F,\mathcal F')$-biharmonic. But the converse is not true. So we prove the generalized Chen's conjecture for $(\mathcal F,\mathcal F')$-biharmonic map. That is, we prove the following theorem
\begin{thm} (cf. Theorem 3.10)  Let $(M,g,\mathcal F)$ be a foliated Riemannian manifold and let $(M',g',\mathcal F')$ be of non-positive transversal sectional curvature $K^{Q'}$, that is, $K^{Q'}\leq 0$. Let $\phi:M\to M'$ be a $(\mathcal F,\mathcal F')$-biharmonic map. Then

 (1) if $M$ is closed, then $\phi$ is automatically $(\mathcal F,\mathcal F')$-harmonic;

(2) if $M$ is complete with $Vol(M)=\infty$ and $\tilde E_{B,2}(\phi)<\infty$, then  $\phi$ is $(\mathcal F,\mathcal F')$-harmonic.

(3) If $M$ is complete with $E_B(\phi)<\infty$ and $\tilde E_{B,2}(\phi)<\infty$, then $\phi$ is $(\mathcal F,\mathcal F')$ -harmonic.
\end{thm}

\begin{rem}  
On foliations, there is another kinds of harmonic map, called {\it transversally harmonic map}, which is a solution of the Eular-Lagrange equation $\tau_b(\phi)=0$ \cite{KW1}.  Also,  the {\it transversally biharmonic map} is defined \cite{CW}, which is not a critical point of the bienergy $\tilde E_{B,2}(\phi)$. Two definitions for harmonic maps are equivalent when the foliation is minimal.  The generalized Chen's conjectures for transversally biharmonic map have been proved  in \cite{JU3,JJ2}.
\end{rem}  
\section{Preliminaries}

Let $(M,g,\mathcal F)$ be a foliated Riemannian
manifold of dimension $n$ with a foliation $\mathcal F$ of codimension $q (=n-p)$ and a bundle-like metric $g$ with respect to $\mathcal F$ \cite{Molino,Tond}.  Let $Q=TM/T\mathcal F$ be the normal  bundle of $\mathcal F$, where  $T\mathcal F$ is the tangent bundle of $\mathcal F$. Let $g_Q$ be the induced metric by $g$ on $Q$, that is, $g_Q = \sigma^* (g|_{T\mathcal F^\perp})$, where $\sigma:Q\to T\mathcal F^\perp$ is the cnonical bundle isomorphism.  Then $g_Q$ is the holonomy invariant metric on $Q$, meaning that  $L_Xg_Q=0$ for $X\in T\mathcal F$, where
$L_X$ is the transverse Lie derivative with respect to $X$.  Let $\nabla^Q$  be the transverse Levi-Civita
connection on the normal bundle $Q$ \cite{Tond,Tond1} and $R^Q$  be the transversal curvature tensor  of $\nabla^Q\equiv\nabla$, which is  defined by $R^Q(X,Y)=[\nabla_X,\nabla_Y]-\nabla_{[X,Y]}$ for any $X,Y\in\Gamma TM$. Let $K^Q$ and ${\rm Ric}^Q $ be the transversal
sectional curvature and transversal Ricci operator with respect to $\nabla$, respectively.
Let $\Omega_B^r(\mathcal F)$ be the space of all {\it basic
$r$-forms}, i.e.,  $\omega\in\Omega_B^r(\mathcal F)$ if and only if
$i(X)\omega=0$ and $L_X\omega=0$ for any $X\in\Gamma T\mathcal F$, where $i(X)$ is the interior product. Then $\Omega^*(M)=\Omega_B^*(\mathcal F)\oplus \Omega_B^*(\mathcal F)^\perp$ \cite{Lop}.   It is well known that $\kappa_B$ is closed, i.e., $d\kappa_B=0$, where  $\kappa_B$ is the basic part of  the mean curvature form $\kappa$ \cite{Lop, PJ} .  
Let $\bar *:\Omega_B^r(\mathcal F)\to \Omega_B^{q-r}(\mathcal F)$ be the star operator  given by
\begin{align*}
\bar *\omega = (-1)^{(n-q)(q-r)} *(\omega\wedge\chi_{\mathcal F}),\quad \omega\in\Omega_B^r(\mathcal F),
\end{align*}
where $\chi_{\mathcal F}$ is the characteristic form of $\mathcal F$ and $*$ is the Hodge star operator associated to $g$.  Let $\langle\cdot,\cdot\rangle$ be the pointwise inner product on $\Omega_B^r(\mathcal F)$, which is given by
\begin{align*}
\langle\omega_1,\omega_2\rangle \nu = \omega_1\wedge\bar * \omega_2,
\end{align*}
where $\nu$ is the transversal volume form such that $*\nu =\chi_{\mathcal F}$. 
 Let $\delta_B :\Omega_B^r (\mathcal F)\to \Omega_B^{r-1}(\mathcal F)$ be the operator defined by
\begin{align*}
\delta_B\omega = (-1)^{q(r+1)+1} \bar * (d_B-\kappa_B \wedge) \bar *\omega,
\end{align*}
where $d_B = d|_{\Omega_B^*(\mathcal F)}$. It is well known  \cite{Park} that $\delta_B$ is the formal adjoint of $d_B$ with respect to the global inner product. That is,
\begin{align*}
\int_M \langle d\omega_1,\omega_2\rangle \mu_M =\int_M \langle \omega_1,\delta_B\omega_2\rangle\mu_M
\end{align*}
for any compactly supported basic forms $\omega_1$ and $\omega_2$, where $\mu_M=\nu\wedge\chi_{\mathcal F}$ is the volume element.

There exists a bundle-like metric  such that  the mean curvature form satisfies $\delta_B\kappa_B=0$ on compact manifolds \cite{DO,MMR,MA}.
 The  basic
Laplacian $\Delta_B$ acting on $\Omega_B^*(\mathcal F)$ is given by
\begin{equation*}
\Delta_B=d_B\delta_B+\delta_B d_B.
\end{equation*}
 Now we define the bundle map $A_Y:\Gamma Q\to \Gamma Q$ for any $Y\in TM$ by
\begin{align}\label{eq1-11}
A_Y s =L_Ys-\nabla_Ys,
\end{align}
where $L_Y s = \pi [Y,Y_s]$ for $\pi(Y_s)=s$. It is well-known \cite{Kamber2} that for any  infitesimal automorphism $Y$ (that is, $[Y,Z]\in \Gamma T\mathcal F$ for
all $Z\in \Gamma T\mathcal F$ \cite{Kamber2})
\begin{align*}
A_Y s = -\nabla_{Y_s}\pi(Y),
\end{align*}
where $\pi:TM\to Q$ is the natural projection and $Y_s$ is the vector field such that $\pi(Y_s)=s$. So $A_Y$ depends only on $\bar Y=\pi(Y)$ and is a linear operator.  Moreover, $A_Y$ extends in an obvious way to tensors of any type on $Q$  \cite{Kamber2}.
Then we
have the generalized Weitzenb\"ock formula on $\Omega_B^*(\mathcal F)$ \cite{JU2}: for any $\omega\in\Omega_B^r(\mathcal  F),$
\begin{align}\label{2-3}
  \Delta_B \omega = \nabla_{\rm tr}^*\nabla_{\rm tr}\omega +
  F(\omega)+A_{\kappa_B^\sharp}\omega,
\end{align}
where $F(\omega)=\sum_{a,b}\theta^a \wedge i(E_b)R^Q(E_b,
 E_a)\omega$ and 
 \begin{align}\label{2-4}
\nabla_{\rm tr}^*\nabla_{\rm tr}\omega =-\sum_a \nabla^2_{E_a,E_a}\omega
+\nabla_{\kappa_B^\sharp}\omega.
\end{align}
 The operator $\nabla_{\rm tr}^*\nabla_{\rm tr}$
is positive definite and formally self adjoint on the space of
basic forms \cite{JU2}. 
  If $\omega$ is a basic 1-form, then $F(\omega)^\sharp
 ={\rm Ric}^Q(\omega^\sharp)$.
 Now, we recall the transversal divergence theorem  on a foliated Riemannian
manifold for later use.
\begin{thm} \label{thm1-1} \cite{Yorozu}
Let $(M,g,\mathcal F)$ be a closed, oriented Riemannian manifold
with a transversally oriented foliation $\mathcal F$ and a
bundle-like metric $g$ with respect to $\mathcal F$. Then for a transversal infinitesimal automorphism $X$,
\begin{equation*}
\int_M \operatorname{div_\nabla}(\pi(X)) \mu_{M}
= \int_M g_Q(\pi(X),\kappa_B^\sharp)\mu_{M},
\end{equation*}
where $\operatorname{div_\nabla} s$
denotes the transversal divergence of $s$ with respect to the
connection $\nabla$.
\end{thm}

\section{$(\mathcal F,\mathcal F')$-harmonic and biharmonic maps on foliations}
Let $\phi :(M,g,\mathcal F) \rightarrow (M', g',\mathcal F')$ be a smooth foliated map, i.e., $d\phi(T\mathcal F)\subset T\mathcal F'$, and $\Omega_B^r(E)=\Omega_B^r(\mathcal F)\otimes E$ be the space of $E$-valued basic $r$-forms, where $E=\phi^{-1}Q'$ is the pull-back bundle on $M$.
We define $d_T\phi:Q \to Q'$ by
\begin{align*}
d_T\phi := \pi' \circ d \phi \circ \sigma.
\end{align*}
Trivially, $d_T\phi\in \Omega_B^1(E)$.   Let $\nabla^\phi$
and $\tilde \nabla$ be the connections on $E$ and
$Q^*\otimes E$, respectively. Then a foliated map $\phi:(M, g,\mathcal F)\to (M', g',\mathcal F')$ is called {\it transversally totally geodesic} if it satisfies
\begin{align}
\tilde\nabla_{\rm tr}d_T\phi=0,
\end{align}
where $(\tilde\nabla_{\rm tr}d_T\phi)(X,Y)=(\tilde\nabla_X d_T\phi)(Y)$ for any $X,Y\in \Gamma Q$. Note that if $\phi:(M,g,\mathcal F)\to (M',g',\mathcal F')$ is transversally totally geodesic with $d\phi(Q)\subset Q'$, then, for any transversal geodesic $\gamma$ in $M$, $\phi\circ\gamma$ is also transversal geodesic.
From now on, we use $\nabla$ instead of all induced connections if we have no confusion.
We define $d_\nabla : \Omega_B^r(E)\to \Omega_B^{r+1}(E)$ by
\begin{align}
d_\nabla(\omega\otimes s)=d_B\omega\otimes s+(-1)^r\omega\wedge\nabla s
\end{align}
for any $s\in \Gamma E$ and $\omega\in\Omega_B^r(\mathcal F)$.
Let $\delta_\nabla$ be a formal adjoint of $d_\nabla$ with respect to the inner product.  
Note that
\begin{align}\label{4-6}
d_\nabla (d_T\phi)=0,\quad\delta_\nabla d_T\phi=-\tau_b (\phi) +d_T\phi(\kappa_B^\sharp),
\end{align}
where $\tau_{b}(\phi)$ is the  {\it transversal tension field}  of $\phi$ defined by
\begin{align}\label{eq3-3}
\tau_{b}(\phi):={\rm tr}_{Q}(\nabla_{\rm tr} d_T\phi).
\end{align}
The Laplacian $\Delta$ on $\Omega_B^*(E)$ is defined by
\begin{align*}
\Delta =d_\nabla \delta_\nabla +\delta_\nabla d_\nabla.
\end{align*}
Moreover, the operator $A_X$ is extended to $\Omega_B^r(E)$ as follows:
\begin{align*}
A_X\Psi&=L_X \Psi -\nabla_X\Psi,
\end{align*}
where $L_X=d_\nabla i(X) +i(X)d_\nabla$ for any $X\in \Gamma TM$ and $i(X)(\omega\otimes s)=i(X)\omega\otimes s$. Hence $\Psi \in\Omega_B^*(E)$ if and only if $i(X)\Psi=0$ and $L_X\Psi=0$ for all $ X\in \Gamma T\mathcal F$.
Then the generalized Weitzenb\"ock type formula (\ref{2-3}) is extended to $\Omega_B^*(E)$ as follows \cite{JJ1}: for any $\Psi\in\Omega_B^r(E)$,
\begin{align}\label{eq4-6}
\Delta \Psi = \nabla_{\rm tr}^*\nabla_{\rm tr} \Psi
 + A_{\kappa_{B}^\sharp} \Psi + F(\Psi), 
\end{align}
where $ \nabla_{\rm tr}^*\nabla_{\rm tr}$ is the operator induced from (\ref{2-4}) and $F(\Psi)=\sum_{a,b=1}^{q}\theta^a\wedge i(E_b) R(E_b,E_a)\Psi$.
Moreover, we have that for any $ \Psi\in\Omega_B^r(E)$,
\begin{align}\label{weitzenbock}
\frac12\Delta_B|\Psi |^{2}
=\langle\Delta \Psi, \Psi\rangle -|\nabla_{\rm tr} \Psi|^2-\langle A_{\kappa_{B}^\sharp}\Psi, \Psi\rangle -\langle F(\Psi),\Psi\rangle.
\end{align}



\subsection{ $(\mathcal F,\mathcal F')$-harmonic maps}
 About this section, see \cite{DT}.  Let $\Omega$ be a compact domain of $M$. Then the {\it transversal energy functional}  of $\phi$ on $\Omega$ is defined by
\begin{align}\label{eq2-4}
E_{B}(\phi;\Omega)={1\over 2}\int_{\Omega} | d_T \phi|^2\mu_{M}.
\end{align}
Then Dragomir and Tommasoli \cite{DT} defined {\it $(\mathcal F,\mathcal F')$-harmonic} if $\phi$ is a critical point of the transversal energy functional $E_{B}(\phi)$. Also, we obtain the first variational formula \cite{DT,JJ1} 
\begin{align}\label{3-12}
{d\over dt}E_{B}(\phi_t;\Omega)\Big|_{t=0}=-\int_{\Omega} \langle \tilde\tau_{b}(\phi),V\rangle \mu_{M},
\end{align}
where $V={d\phi_t\over dt}|_{t=0}$ is the normal variation vector field of a foliated variation $\{\phi_t\}$ of $\phi$ and
\begin{align}
\tilde\tau_b(\phi) :={\tau}_{b}(\phi)-d_T\phi(\kappa_B^\sharp).
\end{align}
From (\ref{3-12}), we have the following \cite{DT}.
\begin{prop}
 A foliated map  $\phi$ is $(\mathcal F,\mathcal F')$-harmonic map if and only if  $\tilde\tau_b(\phi)=0$.
\end{prop}
\begin{rem}  (1) If $\phi:M\to \mathbb R$ is a basic function, then $\tilde\tau_b(\phi) = -\Delta_B\phi$. So $(\mathcal F,\mathcal F')$-harmonic map is a generalization of a basic harmonic function.

(2) On foliated manifold,  there is another kinds of harmonic map, {\it transvesally harmonic map}, which is a solution of the Euler-Lagrange equation $\tau_b(\phi)=0$ by Konderak and Wolak \cite{KW1}. But the transversally harmonic map is not a critical point of the energy functional $E_B(\phi)$.  Two definitions are equivalent when the foliation is minimal.
\end{rem}
Now, we define the {\it transversal Jacobi operator} $J_\phi^T :\Gamma\phi^{-1}Q'\to\Gamma \phi^{-1}Q'$ by
\begin{align}
J^T_\phi (V) = \nabla_{tr}^*\nabla_{tr}V - {\rm tr}_Q R^{Q'}(V,d_T\phi)d_T\phi.
\end{align}
Then $J^T_\phi$ is a formally self-adjoint operator. That is, for any $V,W\in\Gamma \phi^{-1}Q'$,
\begin{align}\label{3-16}
\int_M \langle J^T_\phi(V),W\rangle\mu_M = \int_M \langle V, J^T_\phi(W)\rangle\mu_M.
\end{align}
Also, we have the second variation formula for the transversal energy functional $E_B(\phi)$.
\begin{thm} ( \cite{DT}, The second variation formula)  Let $\phi:(M,g,\mathcal F)\to (M',g',\mathcal F')$ be a $(\mathcal F,\mathcal F')$-harmonic map and let $\{\phi_{s,t}\}$ be  the foliated variation of $\phi$ supported in a compact domain $\Omega$.  Then
\begin{align}\label{3-18}
{\partial^2\over\partial s\partial t} E_B (\phi_{s,t};\Omega) \Big|_{(s,t)=(0,0)} =\int_\Omega \langle J^T_\phi(V),W\rangle \mu_M,
\end{align}
where $V$ and $W$ are the variation vector fields of $\phi_{s,t}$.
\end{thm}
\begin{proof}
Let $V={\partial\phi_{s,t}\over \partial s}\Big|_{(s,t)=(0,0)}$ and $W={\partial\phi_{s,t}\over \partial t}\Big|_{(s,t)=(0,0)}$ be the variation vector fields of $\phi_{s,t}$.   Let $\Phi:M \times (-\epsilon,\epsilon)\times(-\epsilon,\epsilon)\to M'$ be a smooth map, which is defined by $\Phi(x,s,t)=\phi_{s,t}(x)$. Let $\nabla^\Phi$ be the pull-back connection on $\Phi^{-1}Q'$. It is trivial that
$[X,{\partial\over\partial t}]=[X,{\partial\over\partial s}]=0$ for any vector field $X\in TM$.   From  (\ref{3-12}), we have
\begin{align*}
{\partial^2\over\partial s\partial t} E_B(\phi_{s,t};\Omega) = -\int_\Omega \langle {\partial^2\phi_{s,t}\over\partial s\partial t},\tilde\tau_b(\phi_{s,t})\rangle\mu_M -\int_\Omega  \langle {\partial \phi_{s,t}\over\partial t},\nabla^\Phi_{\partial\over\partial s}\tilde\tau_b(\phi_{s,t})\rangle\mu_M.
\end{align*}
At $(s,t)=(0,0)$, the first term vanishes because of  $\tilde\tau_b(\phi)=0$.  So
\begin{align}\label{3-19}
{\partial^2\over\partial s\partial t} E_B(\phi_{s,t};\Omega)\Big|_{(s,t)=(0,0)} =  -\int_\Omega  \langle W,\nabla^\Phi_{\partial\over\partial s}\tilde\tau_b(\phi_{s,t})\Big|_{(s,t)=(0,0)}\rangle\mu_M.
\end{align}
At $x\in M$,  by a straight calculation, we have
\begin{align}\label{3-17}
\nabla^\Phi_{\partial\over\partial s} \tilde\tau_b (\phi_{s,t}) =\sum_a \nabla^\Phi_{E_a}\nabla^\Phi_{E_a} d\Phi({\partial\over\partial s}) -\nabla^\Phi_{\kappa_B^\sharp} d\Phi({\partial\over\partial s}) +\sum_a R^\Phi({d\over dt},E_a)d\Phi(E_a).
\end{align}
Hence at $(s,t)=(0,0)$, we have
\begin{align*}
\nabla^\Phi_{\partial \over\partial s} \tilde\tau_b(\phi_{s,t}) \Big|_{(s,t)=(0,0)} = -\nabla_{tr}^*\nabla_{tr} V + {\rm tr}_Q R^{Q'}(V,d_T\phi)d_T\phi.
\end{align*}
That is, we have
\begin{align}\label{3-21}
\nabla^\Phi _{\partial\over\partial s} \tilde\tau_b(\phi_{s,t})\Big|_{(s,t)=(0,0)} = -J^T_\phi(V).
\end{align}
Hence  the proof of (\ref{3-18}) follows from (\ref{3-19}) and (\ref{3-21}).
\end{proof}
Now, we define  the {\it basic Hessian} $Hess_\phi^T$ of $\phi$ by
\begin{align}
Hess^T_\phi (V,W) = \int_M \langle J^T_\phi(V),W\rangle \mu_M.
\end{align}
Then $Hess^T_\phi(V,W) = Hess^T_\phi (W,V)$ for any $V,W\in \phi^{-1}Q'$.  If  $Hess^T_\phi$  is positive semi-definite, that is, $Hess^T_\phi(V,V)\geq 0$ for any normal vecor field $V$ along $\phi$, then $\phi$ is said to be {\it weakly stable}.  Hence we have the following corollary.
\begin{cor} (\cite{DT}, Stability) Let $M$ be a closed Riemannian manifold and $M'$ be of non-positive transversal sectional curvature. Then any $(\mathcal F,\mathcal F')$-harmonic map $\phi:(M,\mathcal F)\to (M',\mathcal F')$ is weakly stable.
\end{cor}
\begin{rem} For the stability of transversally harmonic map (that is, $\tau_b(\phi)=0$), see \cite[Corollary 4.6]{JU3}. In fact, under the same assumption, a transversally harmonic map is transversally $f$-stable, that is, $\int_M\langle (J_\phi^T-\nabla_{\kappa_B^\sharp})V,V\rangle e^{-f}\mu_M\geq 0$, where $f$ is a basic function such that $\kappa_B=-df$.
\end{rem}

\subsection{$(\mathcal F,\mathcal F')$-biharmonic maps}

We define  the {\it transversal bienergy functional} $\tilde E_{B,2}(\phi)$ on a compact domain $\Omega$  by
 \begin{align}\label{4-5}
\tilde E_{B,2}(\phi;\Omega):=\frac12\int_\Omega |\tilde\tau_b(\phi)|^2\mu_M.
\end{align}
\begin{defn}  A foliated map $\phi:(M,g,\mathcal F)\to (M',g',\mathcal F')$ is said to be {\it $(\mathcal F,\mathcal F')$-biharmonic map} if $\phi$ is a critical point of the transversal bienergy functional $\tilde E_{B,2}(\phi)$.
\end{defn}
\begin{thm} (The first variation formula)  For a foliated map $\phi$,
\begin{align}
{d\over dt}\tilde E_{B,2}(\phi_t;\Omega)\Big|_{t=0} =-\int_\Omega \langle J^T_\phi(\tilde\tau_b(\phi)),V\rangle\mu_M,
\end{align}
where $V={d\phi_t\over dt}\Big|_{t=0}$ is the variation vector field of a foliated variation $\phi_t$ of $\phi$.
\end{thm}
\begin{proof}  
Let $\Phi:M \times (-\epsilon,\epsilon)\to M'$ be a smooth map, which is defined by $\Phi(x,t)=\phi_{t}(x)$. Let $\nabla^\Phi$ be the pull-back connection on $\Phi^{-1}Q'$. It is trivial that
$[X,{\partial\over\partial t}]=0$ for any vector field $X\in TM$. 
 From (\ref{4-5}), we have
 \begin{align}\label{4-7}
 {d\over dt}\tilde E_{B,2}(\phi_t;\Omega)\Big|_{t=0} = \int_\Omega \langle \nabla^\Phi_{d\over dt}\tilde\tau_b(\phi_t) |_{t=0},\tilde\tau_b(\phi)\rangle\mu_M.
 \end{align}
From (\ref{3-16}), (\ref{3-21}) and (\ref{4-7}),  we finish the proof.
 \end{proof}

From the first variation formula for the transversal bienergy functional, we know the following fact.
\begin{prop}  A $(\mathcal F,\mathcal F')$-biharmonic map $\phi$ is a solution of the following equation
\begin{align}\label{3-23}
(\tilde\tau_2)_b(\phi):=J^T_\phi(\tilde\tau_b(\phi))=0.
\end{align}
Here $(\tilde\tau_2)_b(\phi)$ is called the {\it $(\mathcal F,\mathcal F')$-bitension field} of $\phi$. 
\end{prop}
\begin{rem} (1) From Remark 3.2, if $\phi$ is a basic function on $M$, then 
\begin{align*}
(\tilde\tau_2)_b(\phi) = J^T_\phi(\tilde\tau_b( \phi)) = -J^T_\phi (\Delta_B \phi) = -\nabla_{tr}^*\nabla_{tr}(\Delta_B \phi) =\Delta_B^2 \phi.
\end{align*}
So $(\mathcal F,\mathcal F')$-biharmonic map is a generalization of basic biharmonic function.

(2)  A $(\mathcal F,\mathcal F')$-harmonic map is trivially $(\mathcal F,\mathcal F')$-biharmonic map.

(3)  There is another kinds of biharmonic map on foliations, called {\it transversally biharmonic map}, which is a solution of $(\tau_2)_b(\phi):=J_\phi^T(\tau_b(\phi))-\nabla_{\kappa_B^\sharp}\tau_b(\phi) =0$ \cite{JU3}. 
Actually,  transversally biharmonic map is a critical point of the transversal $f$-bienergy functional $E_{2,f}(\phi)$, which is defined by
\begin{align*}
E_{2,f}(\phi) =\frac12\int_M |\tau_b(\phi)|^2 e^{-f} \mu_M,
\end{align*}
where $f$ is a solution of $\kappa_B=-df$.
\end{rem}
\subsection{Generalized Chen's conjecture}
Now, we consider the generalized Chen's conjecture for  $(\mathcal F,\mathcal F')$-biharmonic maps.  
\begin{thm} Let $(M,g,\mathcal F)$ be a foliated Riemannian manifold and let $(M',g',\mathcal F')$ be of non-positive transversal sectional curvature, that is, $K^{Q'}\leq 0$. Let $\phi:M\to M'$ be a $(\mathcal F,\mathcal F')$-biharmonic map. Then

 (1) if $M$ is closed, then $\phi$ is automatically $(\mathcal F,\mathcal F')$-harmonic;

(2) if $M$ is complete with $Vol(M)=\infty$ and $\tilde E_{B,2}(\phi)<\infty$, then  $\phi$ is $(\mathcal F,\mathcal F')$-harmonic;

(3) If $M$ is complete with $E_B(\phi)<\infty$ and $\tilde E_{B,2}(\phi)<\infty$, then $\phi$ is $(\mathcal F,\mathcal F')$ -harmonic.
\end{thm}
\begin{proof}
Let $\phi:M\to M'$ be a $(\mathcal F,\mathcal F')$-biharmonic map.  Then from (\ref{3-23})
\begin{align}\label{4-12}
(\nabla_{tr}^\phi)^*(\nabla_{tr}^\phi)\tilde\tau_b(\phi) -\sum_a R^{Q'}(\tilde\tau_b(\phi),d_T\phi(E_a))d_T\phi(E_a) =0,
\end{align}
where $\{E_a\}$ be a local orthonomal basic frame of $Q$. 
From the generalized Weitzenbock formula (\ref{eq4-6}) and (\ref{weitzenbock}), we have
\begin{align}
\frac12\Delta_B |\tilde\tau_b(\phi)|^2 = \langle \nabla_{tr}^*\nabla_{tr} \tilde\tau_b(\phi),\tilde\tau_b(\phi)\rangle -|\nabla_{tr}\tilde\tau_b(\phi)|^2.
\end{align}
Hence from (\ref{4-12}), we get
\begin{align*}
\frac12\Delta_B |\tilde\tau_b(\phi)|^2  = -|\nabla_{tr}\tilde\tau_b(\phi)|^2 +\sum_a \langle R^{Q'}(\tilde\tau_b(\phi),d_T\phi(E_a))d_T\phi(E_a),\tilde\tau_b(\phi)\rangle.
\end{align*}
That is,
\begin{align}\label{3-26}
|\tilde\tau_b(\phi)|\Delta_B |\tilde\tau_b(\phi)| = |d_B |\tilde\tau_b(\phi)||^2 -|\nabla_{tr}\tilde\tau_b(\phi)|^2 +\sum_a \langle R^{Q'}(\tilde\tau_b(\phi),d_T\phi(E_a))d_T\phi(E_a),\tilde\tau_b(\phi)\rangle.
\end{align}
By the Kato's inequality, that is, $|\nabla_{tr}\tilde\tau_b(\phi)|\geq |d_B|\tilde\tau_b(\phi)||$, and   $K^{Q'}\leq 0$, we have
\begin{align}\label{3-27}
\frac12\Delta_B |\tilde\tau_b(\phi)| \leq 0.
\end{align}
That is, $|\tilde\tau_b(\phi)|$ is basic subharmonic. 
(1) If $M$ is closed, then  $|\tilde\tau_b(\phi)|$ is trivially constant.   From (\ref{3-26}), we have that  for all $a$,
\begin{align}\label{4-15}
\nabla_{E_a}\tilde\tau_b(\phi)=0.
\end{align}
Now, we define the normal vector field $Y$  by
\begin{align*}
Y=\sum_a \langle d_T\phi(E_a),\tilde\tau_b(\phi)\rangle E_a.
\end{align*}
Then  from  (\ref{4-15}), we have
\begin{align}\label{3-30}
{\rm div}_\nabla (Y) = \sum_a\langle \nabla_{E_a}Y,E_a\rangle =\langle \tau_b(\phi),\tilde\tau_b(\phi)\rangle.
\end{align}
So by integrating  (\ref{3-30}) and by using the transversal divergence theorem (Theorem 2.1), we get
\begin{align}
\int_M |\tilde\tau_b(\phi)|^2 \mu_M =0,
\end{align}
which implies that $\tilde\tau_b(\phi)=0$, that is, $\phi$ is the $(\mathcal F,\mathcal F')$-harmonic map.

(2) Let $M$ be a complete Riemannian manifold.  Note that  for any basic 1-form $\omega$,  it is trivial that
$\delta_B\omega =\delta\omega$ and so $\Delta_B f= \Delta f$ for any basic function $f$.  Hence by the Yau's maximum principle \cite[Theorem 3]{YA}, we have following lemma.
\begin{lem}
 If a nonnegative basic function $f$ is basic-subharmonic, that is, $\Delta_B f\leq 0$, with $\int_M f^p <\infty \ (p>1)$, then $f$ is constant.
\end{lem}
Since $\tilde E_{B,2}(\phi)<\infty$, by (\ref{3-27}) and Lemma 3.11,  $|\tilde\tau_b(\phi)|$ is constant. Moreover,  since $Vol(M)=\infty$,  $\int_M |\tilde\tau_B(\phi)|^2\mu_M<\infty$ implies $\tilde\tau_b(\phi)=0$, that is, $\phi$ is $(\mathcal F,\mathcal F')$-harmonic.

(3)  Now we define a basic 1-form $\omega$ on $M$ by
\begin{align}
\omega (X) = \langle d_T\phi (X), \tilde\tau_b(\phi)\rangle
\end{align}
for any normal vector field $X$. By using the Schwartz inequality, we get
\begin{align*}
\int_M |\omega|\mu_M &=\int_M \Big(\sum_a |\omega(E_a)|^2\Big)^{\frac12}\mu_M\\
&=\int_M \Big(\sum_a |\langle d_T\phi(E_a),\tilde\tau_b(\phi)\rangle|^2\Big)^{\frac12}\mu_M\\
&\leq\int_M |d_T\phi| |\tilde\tau_b(\phi)|\mu_M\\
&\leq \Big(\int_M |d_T\phi|^2\mu_M \Big)^{\frac12} \Big(\int_M |\tilde\tau_b(\phi)|^2\mu_M\Big)^{\frac12}\\
&=2\sqrt{E_B(\phi) E_{B,2}(\phi)} <\infty.
\end{align*}
On the other hand, by a straight calculation, we know that
\begin{align}
\delta_B\omega = -|\tilde\tau_b(\phi)|^2.
\end{align}
Since $\int_M |\omega|\mu_M <\infty$ and $\int_M (\delta_B)\omega\mu_M = -\tilde E_{B,2}(\infty) <\infty$, by the Gaffney's theorem \cite{GA}, we know that 
\begin{align}
\int_M |\tilde\tau_b(\phi)|^2\mu_M=-\int_M (\delta_B\omega) \mu_M =-\int_M(\delta\omega)\mu_M=0.
\end{align} 
Hence $\tilde\tau_b(\phi)=0$, that is, $\phi$ is $(\mathcal F,\mathcal F')$-harmonic.
\end{proof}

\begin{rem}  Note that for transversally biharmonic map, we need some conditions that  the transversal Ricci curvature of $M$ is nonnegative and positive at some point  (cf. \cite[Theorem 6.5]{JU3}).
\end{rem}
Now, we study the second variation formula for the transversal bienergy functional $\tilde E_{B,2}(\phi)$. 
\begin{thm}  (The second variation formula) For a  foliated map $\phi:(M,g,\mathcal F)\to (M',g',\mathcal F')$, we  have
\begin{align*}
{d^2\over dt^2} \tilde E_{B,2}(\phi_t;\Omega)\Big|_{t=0} =& -\int_\Omega \langle \nabla_VV,(\tilde\tau_2)_b(\phi)\rangle\mu_M +\int_\Omega |J^T_\phi(V)|^2\mu_M -\int_\Omega \langle R^{Q'}(V,\tilde\tau_b(\phi))\tilde\tau_b(\phi),V\rangle\mu_M\\
&-4\int_M \langle R^{Q'}(\nabla_{tr}V,\tilde\tau_b(\phi)) d_T\phi, V\rangle \mu_M  +\int_\Omega\langle (\nabla_{\tilde\tau_b(\phi)} R^{Q'})(V,d_T\phi)d_T\phi,V\rangle\mu_M\\
&+2\int_\Omega\langle (\nabla_{tr} R^{Q'})(d_T\phi,V)\tilde\tau_b(\phi),V\rangle\mu_M,
\end{align*}
where $V={d\phi_t\over dt}\Big|_{t=0}$ is the normal variation vector field of $\{\phi_t\}$.
\end{thm}
\begin{proof}
 Let $\Phi:M \times (-\epsilon,\epsilon)\to M'$ be a smooth map, which is defined by $\Phi(x,t)=\phi_{t}(x)$. Let $\nabla^\Phi$ be the pull-back connection on $\Phi^{-1}Q'$. It is trivial that
$[X,{\partial\over\partial t}]=0$ for any vector field $X\in TM$. 
 From definition,  we have
\begin{align*}
{d^2\over dt^2} \tilde E_{B,2} (\phi_t;\Omega) =\int_\Omega \langle \nabla^\Phi_{d\over dt}\nabla^\Phi_{d\over dt} \tilde\tau_b(\phi_t),\tilde\tau_b(\phi_t)\rangle\mu_M + \int_\Omega |\nabla^\Phi_{d\over dt}\tilde\tau_b(\phi_t)|^2\mu_M.
\end{align*}
Let $\{E_a\}$ be a local orthonormal basic frame on $Q$ such that $\nabla^\Phi E_a=0$ at $x\in M$. From (\ref{3-17}), we have
\begin{align*}
\nabla^\Phi_{d\over dt}\nabla^\Phi_{d\over dt}\tilde\tau_b(\phi_t)=&\sum_a\nabla^\Phi_{E_a}\nabla^\Phi_{E_a}\nabla^\Phi_{d\over dt}d\Phi({d\over dt}) -\nabla^\Phi_{\kappa_B^\sharp}\nabla^\Phi_{d\over dt} d\Phi({d\over dt})+R^\Phi(\kappa_B^\sharp,{d\over dt})d\Phi({d\over dt})\\
& + \sum_a \nabla^\Phi_{E_a}R^\Phi({d\over dt},E_a)d\Phi({d\over dt}) + \sum_a\nabla_{d\over dt}^\Phi R^\Phi({d\over dt},E_a) d\Phi(E_a)\\
&+\sum_aR^\Phi({d\over dt},E_a)\nabla^\Phi_{E_a} d\Phi({d\over dt}).
\end{align*}
At $t=0$,  since $d\Phi({d\over dt})|_{t=0} = {d\phi_t\over dt}|_{t=0} =V$, we have
\begin{align*}
\nabla^\Phi_{d\over dt}\nabla^\Phi_{d\over dt}\tilde\tau_b(\phi_t)\Big|_{t=0}=&\sum_a\nabla_{E_a}\nabla_{E_a}\nabla_VV -\nabla_{\kappa_B^\sharp}\nabla_VV+R^{Q'}(d_T\phi(\kappa_B^\sharp),V)V\\
& + \sum_a \nabla_{E_a}R^{Q'}(V,d_T\phi(E_a))V + \sum_a\nabla_V R^{Q'}(V,d_T\phi(E_a)) d_T\phi(E_a)\\
&+\sum_aR^{Q'}(V,d_T\phi(E_a))\nabla_{E_a} V.
\end{align*}
By a straight calculation together with the Bianchi identities, we have
\begin{align*}
\sum_a \nabla_{E_a}R^{Q'}(V,d_T\phi(E_a))V&=\sum_a (\nabla_{E_a}R^{Q'})(V,d_T\phi(E_a))V + R^{Q'}(V,\tau_b(\phi))V\\
& + 2\sum_a  R^{Q'}(V,d_T\phi(E_a))\nabla_{E_a}V -\sum_a R^{Q'}(V,\nabla_{E_a}V)d_T\phi(E_a)
\end{align*}
and
\begin{align*}
\sum_a \nabla_V R^{Q'}(V,d_T\phi(E_a))d_T\phi(E_a)=& \sum_a(\nabla_V R^{Q'})(V,d_T\phi(E_a))d_T\phi(E_a)\\
& + \sum_a R^{Q'}(\nabla_VV,d_T\phi(E_a))d_T\phi(E_a)\\
&+\sum_a R^{Q'}(V,\nabla_{E_a}V)d_T\phi(E_a) \\
&+ \sum_a R^{Q'}(V,d_T\phi(E_a))\nabla_{E_a}V.
\end{align*}
By summing the above equations, we have
\begin{align*}
\nabla_{d\over dt}\nabla_{d\over dt}\tilde\tau_b(\phi_t)\Big|_{t=0}=& -J^T_\phi(\nabla_VV ) + R^{Q'}(V,\tilde\tau_b(\phi))V +\sum_a(\nabla_V R^{Q'})(V,d_T\phi(E_a))d_T\phi(E_a)\\
& + \sum_a(\nabla_{E_a}R^{Q'})(V,d_T\phi(E_a))V + 4 \sum_aR^{Q'}(V,d_T\phi(E_a))\nabla_{E_a}V.
\end{align*}
Then by integrating, we get
\begin{align*}
\int_\Omega \langle \nabla^\Phi_{d\over dt}\nabla^\Phi_{d\over dt} \tilde\tau_b(\phi_t)\Big|_{t=0},\tilde\tau_b(\phi)\rangle =& -\int_\Omega \langle J_\phi^T(\nabla_VV),\tilde\tau_b(\phi)\rangle +\int_\Omega\langle R^{Q'}(V,\tilde\tau_b(\phi))V,\tilde\tau_b(\phi)\rangle\\
&+\sum_a\int_\Omega \langle (\nabla_V R^{Q'})(V,d_T\phi(E_a))d_T\phi(E_a),\tilde\tau_b(\phi)\rangle\\
&+\sum_a\int_\Omega\langle (\nabla_{E_a}R^{Q'})(V,d_T\phi(E_a))V,\tilde\tau_b(\phi)\rangle\\
&+4\sum_a\int_\Omega\langle R^{Q'}(V,d_T\phi(E_a))\nabla_{E_a}V,\tilde\tau_b(\phi)\rangle.
\end{align*}
From the second Bianchi identity, we get
\begin{align*}
\langle (\nabla_V R^{Q'})(V,d_T\phi(E_a))d_T\phi(E_a),\tilde\tau_b(\phi)\rangle =& \langle (\nabla_{E_a}R^{Q'})(V,d_T\phi(E_a))V,\tilde\tau_b(\phi)\rangle\\
&+\langle (\nabla_{\tilde\tau_b(\phi)}R^{Q'})(V,d_T\phi(E_a))d_T\phi(E_a),V\rangle.
\end{align*}
From the above equation, we get
\begin{align*}
\int_\Omega \langle \nabla^\Phi_{d\over dt}\nabla^\Phi_{d\over dt} \tilde\tau_b(\phi_t)\Big|_{t=0},\tilde\tau_b(\phi)\rangle =& -\int_\Omega \langle J_\phi^T(\nabla_VV),\tilde\tau_b(\phi)\rangle +\int_\Omega\langle R^{Q'}(V,\tilde\tau_b(\phi))V,\tilde\tau_b(\phi)\rangle\\
&+\sum_a\int_\Omega \langle (\nabla_{\tilde\tau_b(\phi)} R^{Q'})(V,d_T\phi(E_a))d_T\phi(E_a),V\rangle\\
&+2\sum_a\int_\Omega\langle (\nabla_{E_a}R^{Q'})(V,d_T\phi(E_a))V,\tilde\tau_b(\phi)\rangle\\
&+4\sum_a\int_\Omega\langle R^{Q'}(V,d_T\phi(E_a))\nabla_{E_a}V,\tilde\tau_b(\phi)\rangle.
\end{align*}
From the above equation and (\ref{3-21}), by using the curvature properties and  self-adjointness of $J^T_\phi$, the proof follows. 
\end{proof}
\begin{defn} A $(\mathcal F,\mathcal F')$-biharmonic map $\phi:(M,g,\mathcal F)\to (M',g',\mathcal F')$ is said to be {\it weakly stable} if ${d^2\over dt^2}\tilde E_{B,2} (\phi_t)\Big|_{t=0}\geq 0$.
\end{defn}
Now, we consider the generalized Chen's conjecture for $(\mathcal F,\mathcal F')$-biharmonic map when the transversal sectional curvature  of $M'$ is positive, that is, $K^{Q'}>0$.  In case of  $K^{Q'}\leq 0$, see Theorem 3.10. 

Let us recall the transversal stress-energy tensor $S_T(\phi)$ of $\phi$ \cite{CW,JU3}:
\begin{align}
S_T(\phi) =\frac12|d_T\phi|^2 g_Q -\phi^* g_{Q'}.
\end{align}
Note that for any vector field $X\in\Gamma Q$,
\begin{align}
({\rm div}_\nabla S_T(\phi) )(X)= -\langle \tau_b(\phi),d_T\phi(X)\rangle.
\end{align}
If ${\rm div}_\nabla S_T(\phi)=0$, then we say that $\phi$ satisfies the {\it transverse conservation law} \cite{CW}.
\begin{thm}
Let $(M,g,\mathcal F)$ be a  closed foliated Riemannian manifold and $(M',g',\mathcal F')$ be a foliated Riemannian manifold with a positive constant transversal sectional curvature $K^{Q'}$.   If a $(\mathcal F,\mathcal F')$-biharmonic map $\phi:M\to M'$ is weakly stable and satisfies the transverse conservation law,  then $\phi$ is $(\mathcal F,\mathcal F')$-harmonic. 
\end{thm}
\begin{proof}
Let $\phi:M\to M'$ be a $(\mathcal F,\mathcal F')$-biharmonic map, that is, $(\tilde\tau_2)_b(\phi)=0$.
Let $K^{Q'}=c>0$, where $c$ is a positive constant. Then for any $X,Y,Z\in\Gamma Q'$
\begin{align}
R^{Q'}(X,Y)Z =c\{\langle Y,Z\rangle X -\langle X,Z\rangle Y.
\end{align}
 So  $(\nabla_X R^{Q'})(Y,Z) =0$.  Hence if we take $V=\tilde\tau_b(\phi)$ in Theorem 3.13, then
\begin{align*}
{d^2\over dt^2} \tilde E_{B,2}(\phi_t)\Big|_{t=0} =& 
 -4\int_M \langle R^{Q'}(\nabla_{tr}\tilde\tau_b(\phi),\tilde\tau_b) d_T\phi, \tilde\tau_b(\phi)\rangle \mu_M \\
=&-4c\int_M\langle \tilde\tau_b(\phi),d_T\phi\rangle\langle\nabla_{tr}\tilde\tau_b(\phi),\tilde\tau_b(\phi)\rangle \mu_M\\
&+4c\int_M\langle d_T\phi,\nabla_{tr}\tilde\tau_b(\phi)\rangle|\tilde\tau_b(\phi)|^2\mu_M\\
=&-4c\int_M \langle \tau_b(\phi),\tilde\tau_b(\phi)\rangle |\tilde\tau_b(\phi)|^2\mu_M \\
&+4c\sum_a\int_M E_a(\langle d_T\phi(E_a),\tilde\tau_b(\phi)\rangle|\tilde\tau_b(\phi)|^2)\mu_M\\
&-12c\sum_a\int \langle d_T\phi(E_a),\tilde\tau_b(\phi)\rangle \langle \nabla_{E_a}\tilde\tau_b(\phi),\tilde\tau_b(\phi)\rangle \mu_M
\end{align*}
If we choose a normal vector field  $X$ as
\begin{align*}
\langle X,Y\rangle =\langle \tilde\tau_b(\phi),d_T\phi(Y)\rangle |\tilde\tau_b(\phi)|^2
\end{align*}
for any normal vector field $Y$, then
\begin{align*}
{\rm div}_\nabla X = \sum_a E_a (\langle \tilde\tau_b(\phi),d_T\phi(E_a)\rangle |\tilde\tau_b(\phi)|^2 ).
\end{align*}
Hence by the transversal divergence theorem, we have
\begin{align*}
\int\sum_a E_a(\langle d_T\phi(E_a),\tilde\tau_b(\phi)\rangle|\tilde\tau_b(\phi)|^2)\mu_M &= \int {\rm div}_\nabla (X)\mu_M = \int \langle X,\kappa_B^\sharp\rangle\mu_M \\
&=\int \langle d_T\phi(\kappa_B^\sharp),\tilde\tau_b(\phi)\rangle |\tilde\tau_b(\phi)|^2 \mu_M.
\end{align*}
Combining the above equations, we have
\begin{align*}
{d^2\over dt^2} \tilde E_{B,2}(\phi_t)\Big|_{t=0} =&-4c\int_M |\tilde\tau_b(\phi)|^4\mu_M \\
&-12c\sum_a\int_M \langle d_T\phi(E_a),\tilde\tau_b(\phi)\rangle \langle \nabla_{E_a}\tilde\tau_b(\phi),\tilde\tau_b(\phi)\rangle \mu_M
\end{align*}
Since $\phi$ satisfies the transverse conservation law, that is, $({\rm div}_\nabla S_T(\phi))(X)=0$ for any $X$,  we have
\begin{align*}
\langle \tau_b(\phi),d_T\phi(E_a)\rangle = ({\rm div}_\nabla S_T(\phi))(E_a)=0.
\end{align*}
So  if we choose the bundle-like metric such that $\delta_B\kappa_B=0$, then
\begin{align*}
\int_M&\sum_a \langle\tilde\tau_b(\phi),d_T\phi(E_a)\rangle \langle\nabla_{E_a}\tilde\tau_b(\phi),\tilde\tau_b(\phi)\rangle\mu_M \\&= -\int_M\sum_a\langle d_T\phi(\kappa_B^\sharp),d_T\phi(E_a)\rangle \langle \nabla_{E_a}\tilde\tau_b(\phi),\tilde\tau_b(\phi)\rangle\mu_M\\
&=-\int _M\langle\nabla_{\kappa_B^\sharp} \tilde\tau_b(\phi),\tilde\tau_b(\phi)\rangle \mu_M\\
&=-\frac12\int_M \langle\delta_B\kappa_B, |\tilde\tau_b(\phi)|\rangle\mu_M\\
& =0.
\end{align*}
Hence we have
\begin{align*}
{d^2\over dt^2} \tilde E_{B,2}(\phi_t)\Big|_{t=0} =-4c\int  |\tilde\tau_b(\phi)|^4\mu_M.
\end{align*}
Since $\phi$ is weakly stable and $c>0$,  we have $\tilde\tau_b(\phi)=0$, that is, $\phi$ is $(\mathcal F,\mathcal F')$-harmonic.
\end{proof}
\begin{rem} The generalized Chen's conjectures for the transversally biharmonic map have been studied in \cite{JU3,JJ2} under some additional conditions such that the transversal Ricci curvature of $M$ is nonnegative.
\end{rem}

\end{document}